\documentclass{amsart}

\usepackage{amssymb}
\usepackage{nicefrac}
\usepackage{hyperref}
\newcommand*{\mailto}[1]{\href{mailto:#1}{\nolinkurl{#1}}}
\newcommand{\arxiv}[1]{\href{http://arxiv.org/abs/#1}{arXiv:#1}}
\newtheorem{theorem}{Theorem}[section]

\newtheorem{lemma}[theorem]{Lemma}
\newtheorem{corollary}[theorem]{Corollary}
\newtheorem{remark}[theorem]{Remark}
\newtheorem{hypothesis}[theorem]{Hypothesis}

\newcommand{\R}{{\mathbb R}}
\newcommand{\N}{{\mathbb N}}
\newcommand{\Z}{{\mathbb Z}}
\newcommand{\C}{{\mathbb C}}

\newcommand{\OO}{\mathcal{O}}
\newcommand{\oo}{o}

\newcommand{\nn}{\nonumber}

\newcommand{\ol}{\overline}
\newcommand{\ti}{\tilde}

\newcommand{\spr}[2]{\langle #1 , #2 \rangle}
\newcommand{\dbspr}[2]{[ #1 , #2 ]}

\newcommand{\E}{\mathrm{e}}
\newcommand{\I}{\mathrm{i}}

\newcommand{\im}{\mathrm{Im}}
\newcommand{\re}{\mathrm{Re}}
\newcommand{\dom}{\mathfrak{D}}

\DeclareMathOperator{\lspan}{span}

\newcommand{\ceil}[1]{\lceil#1 \rceil}
\newcommand{\lz}{\ell^2(\Z)}

\newcommand{\eps}{\varepsilon}

\newcommand{\sig}{\sigma}
\newcommand{\lam}{\lambda}
\numberwithin{equation}{section}

\begin{document}

\title[Singular Weyl--Titchmarsh--Kodaira theory]{Singular Weyl--Titchmarsh--Kodaira theory for Jacobi operators}

\author[J.\ Eckhardt]{Jonathan Eckhardt}
\address{Faculty of Mathematics\\ University of Vienna\\
Nordbergstrasse 15\\ 1090 Wien\\ Austria}
\email{\mailto{jonathan.eckhardt@univie.ac.at}}
\urladdr{\url{http://homepage.univie.ac.at/jonathan.eckhardt/}}

\author[G.\ Teschl]{Gerald Teschl}
\address{Faculty of Mathematics\\ University of Vienna\\
Nordbergstrasse 15\\ 1090 Wien\\ Austria\\ and International
Erwin Schr\"odinger
Institute for Mathematical Physics\\ Boltzmanngasse 9\\ 1090 Wien\\ Austria}
\email{\mailto{Gerald.Teschl@univie.ac.at}}
\urladdr{\url{http://www.mat.univie.ac.at/~gerald/}}

\thanks{Oper. Matrices {\bf 7}, 695--712 (2013)}
\thanks{{\it Research supported by the Austrian Science Fund (FWF) under Grant No.\ Y330}}

\keywords{Jacobi operators, inverse spectral theory, discrete spectra}
\subjclass[2010]{Primary 47B36, 34B20; Secondary 47B39, 34A55}

\begin{abstract}
We develop singular Weyl--Titchmarsh--Kodaira theory for Jacobi operators. In particular, we establish
existence of a spectral transformation as well as local Borg--Marchenko and Hochstadt--Liebermann type
uniqueness results.
\end{abstract}

\maketitle

\section{Introduction}
\label{sec:int}

Classical Weyl--Titchmarsh--Kodaira theory was originally developed for one-dimensional Schr\"odinger operators
with one regular endpoint. Moreover, it has been shown by Kodaira \cite{ko}, Kac \cite{ka} and more recently by Fulton \cite{ful08},
Gesztesy and Zinchenko \cite{gz}, Fulton and Langer \cite{fl}, Kurasov and Luger \cite{kl}, and Kostenko, Sakhnovich, and Teschl
\cite{kt}, \cite{kst}, \cite{kst2}, \cite{kst3}, that many aspects of this classical theory still can be established at a
singular endpoint. It has recently proven to be a powerful tool for inverse spectral theory for these operators
and further refinements were given by us in \cite{je}, \cite{egnt}, \cite{et}, \cite{kt}.

Of course the analog of classical Weyl--Titchmarsh--Kodaira theory is also
a basic ingredient for inverse spectral theory for Jacobi operators \cite{tjac}. The purpose of the present paper is
to extend singular Weyl--Titchmarsh--Kodaira theory to the case of Jacobi operators. While the overall approach
generalizes in a straightforward manner, there are some significant differences in the proofs of our main inverse uniqueness result:
Theorem~\ref{thmbm} and Theorem~\ref{thmdBuniqS}. This is related to the fact that in the case of Jacobi operators one
needs to determine two coefficients $a(n)$ and $b(n)$ while in the case of one-dimensional Schr\"odinger operators
there is only one coefficient $q(x)$. In fact, it is well known that in the case of general Sturm--Liouville operators with
three coefficients $r(x)$, $p(x)$, and $q(x)$ the operator can only be determined up to a Liouville transform (cf.\ e.g.\ \cite{ben1}, \cite{je2}).

As our main result we first prove a local Borg--Marchenko result (Theorem~\ref{thmbm}) which generalizes the
classical result whose local version was first established by Gesztesy, Kiselev, and Makarov \cite{gkm}
(see also Weikard \cite{w}). Moreover, we show that in the
case of purely discrete spectra the spectral measure uniquely determines the operator (Theorem~\ref{thmSpectFuncDisc})
and use this to establish a general Hochstadt--Liebermann-type uniqueness result (Theorem~\ref{thmHL}). Finally
we use the connection with de Branges spaces of entire functions in order to give another general criterion when
the spectral measure uniquely determines the operator (Theorem~\ref{thmdBuniqS}).

\section{Singular Weyl--Titchmarsh--Kodaira theory}
\label{sec:swm}

We will be concerned with operators in $\lz$ associated with the
difference expression
\begin{equation}
(\tau f)(n) = a(n) f(n+1) + a(n-1) f(n-1) +b(n) f(n), \quad n\in\Z,
\end{equation}
where the sequences $a$, $b \in \ell(\Z)$ satisfy
\begin{hypothesis} \label{hab}
Suppose
\begin{equation}
a(n)>0, \quad b(n) \in \R, \quad n \in \Z.
\end{equation}
\end{hypothesis}
If $\tau$ is limit point ($l.p.$) at both $\pm\infty$ (cf., e.g., \cite{tjac}), then
$\tau$ gives rise to a unique self-adjoint operator $H$ when
defined maximally. Otherwise, we need to fix a boundary condition at each
endpoint where $\tau$ is limit circle ($l.c.$) (cf., e.g., \cite{tjac}).
Throughout this paper we denote by $u_\pm(z,\cdot\,)$, $z \in \C$, nontrivial solutions
of $\tau u = z u$ which satisfy the boundary condition at $\pm\infty$ (if any) with
$u_\pm(z,\cdot\,) \in \ell^2_\pm(\Z)$, respectively. Here $\ell^2_\pm(\Z)$ denotes the
sequences in $\ell(\Z)$ being $\ell^2$ near $\pm\infty$. The solution $u_\pm(z,\cdot\,)$
might not exist for $z \in \R$, but if it exists it is
unique up to a constant multiple.

Picking a fixed $z_0 \in \C \backslash \R$ we can characterize $H$ by
\begin{equation}
H : \dom(H)\subseteq\ell^2(\Z)  \to  \lz, \qquad f \mapsto \tau f,
\end{equation}
where the domain of $H$ is explicitly given by
\begin{equation}
\dom(H) = \{ f \in \lz \,|\, \tau f \in \lz, \: \lim_{n \to \pm\infty}
W(u_\pm(z_0),f)(n) = 0\}
\end{equation}
and
\begin{equation}
W(f,g)(n) = a(n) \Big( f(n)g(n+1) - f(n+1)g(n) \Big), \quad n\in\Z
\end{equation}
denotes the (modified) Wronskian. The boundary condition at $\pm\infty$ imposes
no additional restriction on $f$ if $\tau$ is $l.p.$ at $\pm\infty$ and can hence be
omitted in this case. 
 We will also consider the operators $H_\pm$ which are obtained by restricting $H$ to $\pm\N$, respectively with a Dirichlet boundary condition at
$0$.

For the rest of this section we will closely follow the presentation from \cite{kst2}. Most proofs can be done literally following the
arguments in \cite{kst2} and hence we will omit them here. Our first ingredient to define an analogous singular Weyl function
at $-\infty$ is a system of entire solutions $\theta(z,n)$ and $\phi(z,n)$ such that $\phi(z,n)$ lies in the domain of
$H$ near $-\infty$ and whose Wronskian satisfies $W(\theta(z),\phi(z))=1$. To this end we require the following
hypothesis which turns out necessary and sufficient for such a system of solutions to exist.

\begin{hypothesis}\label{hyp:gen}
Suppose that the spectrum of $H_-$ is purely discrete.
\end{hypothesis}

Note that this hypothesis is for example satisfied if $a(n)=1$ and $b(n) \to +\infty$ as $n\to -\infty$
 or if $a(n) \to 0$ as $n\to -\infty$.

\begin{lemma}\label{lem:pt}
The following properties are equivalent:
\begin{enumerate}
\item The spectrum of $H_-$ is purely discrete.
\item There is a real entire solution $\phi(z,n)$, which is non-trivial and lies in the domain of $H$
near $-\infty$ for each $z\in\C$.
\item There are real entire solutions $\theta(z,n)$, $\phi(z,n)$ with $W(\theta,\phi)=1$, such that
 $\phi(z,n)$ is non-trivial and lies in the domain of $H$ near $-\infty$ for each $z\in\C$.
\end{enumerate}
\end{lemma}

Given such a system of real entire solutions $\theta(z,n)$ and $\phi(z,n)$ we define the singular Weyl function
\begin{equation}\label{defM}
M(z) = -\frac{W(\theta(z),u_+(z))}{W(\phi(z),u_+(z))}
\end{equation}
such that the solution lying in the domain of $H$ near $+\infty$ is given by
\begin{equation}
u_+(z,n)= \alpha(z) \big(\theta(z,n) + M(z) \phi(z,n)\big), \quad n\in\Z,
\end{equation}
where $\alpha(z)= - W(\phi(z),u_+(z))$.
It is immediate from the definition that the singular Weyl function $M(z)$ is analytic in $\C\backslash\R$
and satisfies $M(z)=M(z^*)^*$. Rather than $u_+(z,n)$ we will use
\begin{equation}\label{defpsi}
\psi(z,n)= \theta(z,n) + M(z) \phi(z,n), \quad n\in\Z.
\end{equation}
Following literally the argument in \cite[Lem.~3.2]{kst2}, one infers that associated with $M(z)$ is a corresponding
spectral measure given by the Stieltjes--Liv\v{s}i\'{c} inversion formula
\begin{equation}\label{defrho}
\frac{1}{2} \left( \rho\big((x_0,x_1)\big) + \rho\big([x_0,x_1]\big) \right)=
\lim_{\eps\downarrow 0} \frac{1}{\pi} \int_{x_0}^{x_1} \im\big(M(x+\I\eps)\big) dx.
\end{equation}

\begin{theorem}
Define
\begin{equation}\label{eqhatf}
\hat{f}(\lam) = \lim_{n\to\infty} \sum_{m=-\infty}^n \phi(\lam,m) f(m),
\end{equation}
where the right-hand side is to be understood as a limit in $L^2(\R,d\rho)$. Then the map
\begin{equation}
U: \lz \to L^2(\R,d\rho), \qquad f \mapsto \hat{f},
\end{equation}
is unitary with inverse given by
\begin{equation}\label{Uinv}
f(n) = \lim_{r\to\infty} \int_{-r}^r \phi(\lam,n) \hat{f}(\lam) d\rho(\lam),
\end{equation}
where again the right-hand side is to be understood as a limit in $\lz$.
Moreover, $U$ maps $H$ to multiplication with the independent variable in $L^2(\R,d\rho)$.
\end{theorem}

\begin{corollary}\label{corsup}
The following sets
\begin{align} \nn
\Sigma_{ac} &= \{\lam\in\R \,|\, 0<\limsup_{\eps\downarrow 0} \im(M(\lam+\I\eps)) < \infty\},\\
\Sigma_s &= \{\lam\in\R \,| \limsup_{\eps\downarrow 0}\im(M(\lam+\I\eps)) = \infty\},\\ \nn
\Sigma_p &= \{\lam\in\R \,| \lim_{\eps\downarrow 0} \eps\im(M(\lam+\I\eps))>0 \},\\
\Sigma &= \Sigma_{ac} \cup \Sigma_s = \{\lam\in\R \,|\, 0<\limsup_{\eps\downarrow 0} \im(M(\lam+\I\eps))\}
\end{align}
are minimal supports for $\rho_{ac}$, $\rho_s$, $\rho_{pp}$, and $\rho$, respectively.
In fact, we could even restrict ourselves to values of $\lam$, where the $\limsup$ is a $\lim$ (finite or infinite).

Moreover, the spectrum of $H$ is given by the closure of $\Sigma$,
\begin{equation}
\sig(H) = \ol{\Sigma},
\end{equation}
the set of eigenvalues is given by
\begin{equation}
\sig_p(H) = \Sigma_p,
\end{equation}
and the absolutely continuous spectrum of $H$ is given by the essential closure
of $\Sigma_{ac}$,
\begin{equation}
\sig(H_{ac}) = \ol{\Sigma}_{ac}^{ess}.
\end{equation}
Hereby recall $ \ol{\Omega}^{ess} = \{ \lam\in\R \,|\, |(\lam-\eps,\lam+\eps)\cap \Omega|>0
\mbox{ for all } \eps>0\}$, where $|\Omega|$ denotes the Lebesgue measure of a Borel set $\Omega$.
\end{corollary}

\begin{lemma}
We have
\begin{equation}
(U \partial_z^k G(z,n,\cdot\,))(\lam) = \frac{k! \phi(\lam,n)}{(\lam-z)^{k+1}}
\end{equation}
for every $k\in\N_0$, and every $z\in\C\setminus\sig(H)$, where $G(z,n,m)=\spr{\delta_n}{(H-z)^{-1}\delta_m}$ is the
Green's function of $H$.
\end{lemma}

\begin{remark}\label{rem:uniq}
It is important to point out that a fundamental system $\theta(z,n)$, $\phi(z,n)$ of solutions is not unique and any other such system is given by
\[
\ti{\theta}(z,n) = \E^{-g(z)} \theta(z,n) - f(z) \phi(z,n), \qquad
\ti{\phi}(z,n) = \E^{g(z)} \phi(z,n),
\]
where $g(z)$, $f(z)$ are entire functions with $f(z)$ real and $g(z)$ real modulo $\I\pi$.
The singular Weyl functions are related via
\[
\ti{M}(z) = \E^{-2g(z)} M(z) + \E^{-g(z)}f(z)
\]
and the corresponding spectral measure is given by
\[
d\ti{\rho}(\lam) = \E^{-2g(\lam)} d\rho(\lam).
\]
Hence the two measures are mutually absolutely continuous and the associated spectral
transformations just differ by a simple rescaling with the positive function $\E^{-2g(\lam)}$.
\end{remark}

Next, the following integral representation shows that $M(z)$ can be reconstructed from $\rho$ up to an entire function.

\begin{theorem}[\cite{kst2}]\label{IntR}
Let $M(z)$ be a singular Weyl function and $\rho$ its associated spectral measure. Then there exists
an entire function $g(z)$ such that $g(\lam)\ge 0$ for $\lam\in\R$ and $\E^{-g(\lam)}\in L^2(\R, d\rho)$.

Moreover, for any entire function $\hat{g}(z)$ such that $\hat{g}(\lam)>0$ for $\lam\in\R$ and $(1+\lam^2)^{-1} \hat{g}(\lam)^{-1}\in L^1(\R, d\rho)$
(e.g.\ $\hat{g}(z)=\E^{2g(z)}$) we have the integral representation
\begin{equation}\label{Mir}
M(z) = E(z) + \hat{g}(z) \int_\R \left(\frac{1}{\lam-z} - \frac{\lam}{1+\lam^2}\right) \frac{d\rho(\lam)}{\hat{g}(\lam)},
\qquad z\in\C\backslash\sig(H),
\end{equation}
where $E(z)$ is a real entire function.
\end{theorem}

\begin{remark}\label{rem:herg}
Choosing a real entire function $g(z)$ such that $\exp(-2g(\lam)) \in L^1(\R, d\rho)$ we see that
\begin{equation}
M(z) = \E^{2g(z)} \int_\R \frac{1}{\lam-z} \E^{-2g(\lam)}d\rho(\lam) - E(z)
\end{equation}
for some real entire function $E(z)$. 
Hence if we choose $f(z) = \exp(-g(z)) E(z)$ and switch to a new system of solutions as in Remark~\ref{rem:uniq},
we see that the new singular Weyl function is a Herglotz--Nevanlinna function
\begin{equation}
\ti{M}(z) = \int_\R \frac{1}{\lam-z} \E^{-2g(\lam)}d\rho(\lam).
\end{equation}
\end{remark}

As a final ingredient we will need the following simple lemma on high energy asymptotics of our real entire solution $\phi(z,n)$.

\begin{lemma}\label{lemPhiAs}
If $\phi(z,n)$ is a real entire solution which lies in the domain of $H$ near $-\infty$, then for every $n$, $\tilde{n}\in\Z$
\begin{equation}\label{IOphias}
\phi(z,n) = \phi(z,\tilde{n}) \begin{cases}
 z^{n-\tilde{n}} \prod\limits_{m=\tilde{n}}^{n-1}a(m)^{-1} (1+ \OO(z^{-1})), & n\ge \tilde{n}\\
z^{n-\tilde{n}}\prod\limits_{m=n}^{\tilde{n}-1}a(m) (1+ \OO(z^{-1})),& n<\tilde{n}
\end{cases}
\end{equation}
as $|z|\to\infty$ along any non-real ray.
\end{lemma}

\begin{proof}
This follows from induction using \cite[Lemma~6.6]{tjac},
\[
m_-(z,n)= -\frac{\phi(z,n-1)}{a(n-1)\phi(z,n)} = -\frac{1}{z} + \OO(z^{-2})
\]
as $|z|\rightarrow\infty$ along non-real rays. 
\end{proof}

\section{Exponential growth rates}
\label{sec:egr}

It turns out that the real entire fundamental system $\theta(z,n)$, $\phi(z,n)$ from Section~\ref{sec:swm} is not sufficient for the proofs of our inverse uniqueness results. 
To this end we will need information on the growth order of the functions $\theta(\,\cdot\,,n)$ and $\phi(\,\cdot\,,n)$.
Our presentation in this section will closely follow \cite[Section~3]{et}.

We will say a real entire solution $\phi(z,n)$ is of growth order at most $s\geq 0$ if the entire functions $\phi(\,\cdot\,,n)$ and $\phi(\,\cdot\,,n+1)$ are of growth order at most $s$ for 
one (and hence for all) $n\in\Z$.
 Our first aim is to extend Lemma~\ref{lem:pt} and to establish the connection between the growth order of $\phi(z,n)$ and the convergence exponent
of the spectrum. We begin by recalling some basic notation and refer to the classical book by Levin \cite{lev} for proofs and further
background.

Given some discrete set $S\subseteq \C$, the number
\begin{equation}
 \inf\biggr\lbrace s\geq0 \,\biggr|\, \sum_{\mu\in S} \frac{1}{1+|\mu|^s}<\infty \biggr\rbrace \in [0,\infty],
\end{equation}
is called the convergence exponent of $S$. Moreover, the smallest integer $p\in\N_0$ for which
\begin{equation}
 \sum_{\mu\in S} \frac{1}{1+|\mu|^{p+1}}<\infty
\end{equation}
will be called the genus of $S$. Introducing the elementary factors
\begin{equation}
 E_p(\zeta,z) = \left(1-\frac{z}{\zeta}\right) \exp\left(\sum_{k=1}^p\frac{1}{k} \frac{z^k}{\zeta^k}\right), \quad z\in\C,
\end{equation}
if $\zeta\not=0$ and $E_p(0,z)=z$, we recall that the product $\prod_{\mu\in S} E_p(\mu,z)$
converges uniformly on compact sets to an entire function of growth order $s$, where $s$ and $p$ are the
convergence exponent and genus of $S$, respectively.

\begin{theorem}\label{thm:IOphiev}
For each $s\geq 0$ the following properties are equivalent:
\begin{enumerate}
\item The spectrum of $H_-$ is purely discrete and has convergence exponent at most $s$.
\item There  is a real entire solution $\phi(z,n)$ of growth order at most $s$ which is non-trivial and lies in the domain of $H$ near $-\infty$ for each $z\in\C$.
\end{enumerate} 
\end{theorem}

\begin{proof}
Suppose the spectrum of $H_{-}$ is purely discrete and has convergence exponent at most $s$. The same then
holds true for the spectrum of the operator $H_{-}'$ which is obtained by restricting $H$ to $-\N_0$ with a Dirichlet boundary condition at $1$.
 The spectra of these operators will be denoted with 
\begin{equation}
\sig(H_{-}) = \{ \mu_j \}_{j\in N} \quad\text{and}\quad \sig(H_{-}') = \{ \nu_{j-1} \}_{j\in N},
\end{equation}
where the index set $N$ is either $\N$ or $\Z$. Note that the eigenvalues $\mu_j$, $\nu_{j-1}$, $j\in N$ are precisely the zeros
of $\phi(\,\cdot\,,0)$ and $\phi(\,\cdot\,,1)$, respectively. Also recall that both spectra are interlacing
\begin{equation}\label{IOinterlacing}
\nu_{j-1} < \mu_j < \nu_j, \qquad j\in N,
\end{equation}
and that Krein's theorem \cite[Thm.~27.2.1]{lev} states
\begin{equation}\label{eqKrein}
m_-(z,1) = C \prod_{j\in N} \frac{E_0(\mu_j,z)}{E_0(\nu_{j-1},z)}
\end{equation}
for some real constant $C\not=0$. 
 Now consider the real entire functions
\[
\alpha(z) = \prod_{j\in N} E_p\left(\nu_{j-1},z\right) \quad\text{and}\quad \ti{\beta}(z) = \prod_{j\in N} E_p\left(\mu_{j},z\right),
\]
where $p\in\N_0$ is the genus of these sequences. 
Then $\alpha(z)$ and $\ti{\beta}(z)$ are of growth order at most $s$ by Borel's theorem (see \cite[Thm.~4.3.3]{lev}).
Next note that $m_-(z,1) = \E^{h(z)} \ti{\beta}(z) \alpha(z)^{-1}$ for some entire function $h(z)$ since the right-hand side
has the same poles and zeros as $m_-(z,1)$. Comparing this with Krein's formula
\eqref{eqKrein} we obtain that $h(z)$ is in fact a polynomial of degree at most $p$:
\[
h(z) = \sum_{k=1}^p \frac{z^k}{k} \sum_{j\in N} \left( \frac{1}{\nu_{j-1}^k} - \frac{1}{\mu_j^k}\right) +\ln(C), \quad z\in\C,
\]
where the sums converge absolutely by our interlacing assumption.
In particular, $\beta(z) = -a(0) m_-(z,1)\alpha(z)= - a(0) \E^{h(z)} \ti{\beta}(z)$ is of growth order at most $s$ as well.
Hence the solutions $\phi(z,n)$ with $\phi(z,1) = \alpha(z)$ and $\phi(z,0) = \beta(z)$, $z\in\C$ lie
in the domain of $H$ near $-\infty$ and are of growth order at most $s$. 

Conversely let $\phi(z,n)$ be a real entire solution of growth order at most $s$ which lies in the domain of $H$ near $-\infty$.
Then since $m_-(z,0) = -\phi(z,-1) \phi(z,0)^{-1}a(-1)^{-1}$, the spectrum of $H_-$ is purely discrete and coincides with the zeros of $\phi(\,\cdot\,,0)$.
Now since $\phi(\,\cdot\,,0)$ is of growth order at most $s$, its zeros are of convergence exponent at most $s$.
\end{proof}

Given a real entire solution $\phi(z,n)$ of growth order $s\geq 0$ we are not able to prove the existence of a second solution of the same growth order.
Hence we recall the following lemma which provides a criterion to ensure existence of a second solution $\theta(z,n)$ of growth order arbitrarily close to $s$.

\begin{lemma}[\cite{et}]
Suppose $\phi(z,n)$ is a real entire solution of growth order $s\geq 0$ and let $\eps>0$, $\tilde{n}\in\Z$. 
Then there is a real entire second solution $\theta(z,n)$ with 
\begin{equation}
 |\theta(z,\tilde{n})| + |\theta(z,\tilde{n}+1)| \leq B \E^{A |z|^{s+\eps}}, \quad z\in\C
\end{equation}
for some constants $A$, $B>0$ and $W(\theta,\phi)=1$ if and only if 
\begin{equation}\label{eqnSSphiboundlow}
   |\phi(z,\tilde{n})| + |\phi(z,\tilde{n}+1)| \geq b \E^{-a|z|^{s+\eps}}, \quad z\in\C
\end{equation}
for some constants $a$, $b>0$.
\end{lemma}

This enables us to provide a sufficient condition for a second solution of order $s+\eps$ to exist, in terms of the interlacing zeros of $\phi(\,\cdot\,,\tilde{n})$
and $\phi(\,\cdot\,,\tilde{n}+1)$, which we denote by $\lbrace \mu_j\rbrace_{j\in N}$ and $\lbrace \nu_{j-1}\rbrace_{j\in N}$ respectively, where $N$ is either $\N$ or $\Z$.

\begin{lemma}[\cite{et}]\label{lem:cor}
 Suppose $\phi(z,n)$ is a real entire solution of growth order $s\geq 0$ and that for some $r>0$ all but finitely many of the discs given by
\begin{equation}\label{IOestmunu}
|z-\mu_j|<|\mu_j|^{-r} \quad\text{and}\quad |z-\nu_{j-1}|<|\nu_{j-1}|^{-r}, \quad j\in N,
\end{equation}
are disjoint. Then for every $\eps>0$ there is a real entire second solution $\theta(z,n)$ with growth order at most $s+\eps$ and $W(\theta,\phi)=1$.
\end{lemma}

\begin{remark}\label{rem:uniqExp}
By the Hadamard product theorem~\cite[Thm.~4.2.1]{lev}, a solution $\phi(z,n)$ of growth order $s\geq 0$ is unique up to a factor $\E^{g(z)}$,
for some polynomial $g(z)$ real modulo $\I\pi$ and of degree at most $p$, where $p\in\N_0$ is the genus of the eigenvalues of $H_{-}$.
A solution $\theta(z,n)$ of growth order at most $s$ is unique only up to $f(z) \phi(z,n)$, where $f(z)$ is a real entire function of growth order at most $s$.
\end{remark}

Finally, observe that under the assumptions in this section one can use $\hat{g}(z)=\exp(z^{2\ceil{(p+1)/2}})$ in Theorem~\ref{IntR}.
Moreover, under the additional assumption that $H$ is bounded from below, one can also use $\hat{g}(z)=\exp(z^{p+1})$.

\section{A local Borg--Marchenko uniqueness result}
\label{sec:lbmt}

The preparations from the previous sections now enables us to prove a local Borg--Marchenko uniqueness result for the singular
Weyl function, again extending the results from \cite{et} to the case of Jacobi operators.

\begin{lemma}[\cite{kst2}]\label{lemAsymM}
The singular Weyl function $M(z)$ and the Weyl solution $\psi(z,n)$ defined in~\eqref{defpsi}
have the following asymptotics
\begin{align}\label{asymM}
M(z) &= -\frac{\theta(z,n)}{\phi(z,n)} + \OO\left(\frac{1}{z\phi(z,n)^2}\right),\\ \label{asympsi}
\psi(z,n) &= \frac{-1}{z \phi(z,n)} \left( 1 + \OO\left(\frac{1}{z}\right) \right)
\end{align}
as $|z|\to\infty$ in any sector $|\im(z)| \geq \delta\, |\re(z)|$.
\end{lemma}

\begin{proof}
This follows by solving the well-known asymptotical formula (\cite[Thm.~6.2]{tjac})
for the diagonal of Green's function
\[
G(z,n,n) = \phi(z,n) \psi(z,n) = -\frac{1}{z} + \OO(z^{-2})
\]
for $\psi(z,n)$ and $M(z)$.
\end{proof}

Note that \eqref{asymM} shows that the asymptotics of $M(z)$ immediately follow from the asymptotics for the solutions $\theta(z,n)$ and $\phi(z,n)$.
Furthermore, the leading asymptotics depend only on the values of the sequences $a$, $b$ near the endpoint $-\infty$ (and on the choice of $\theta(z,n)$ and $\phi(z,n)$). 
The following Borg--Marchenko type uniqueness result shows that the converse is also true.

To state this theorem let $\{a_0,b_0\}$ and $\{a_1,b_1\}$ be two sets of coefficients on $\Z$ satisfying Hypothesis~\ref{hab}.
By $H_0$ and $H_1$ we denote some corresponding self-adjoint operators with separated boundary conditions.
Furthermore, for $j=0,1$ let $\theta_j(z,n)$, $\phi_j(z,n)$ be some real entire fundamental system of solutions with
$W_j(\theta_j,\phi_j)=1$ such that $\phi_j(z,n)$ lie in $\ell^2(\Z)$ near $-\infty$ and satisfy the boundary condition of $H_j$ there (if any). 
The associated singular Weyl functions are denoted by $M_0(z)$ and $M_1(z)$. We will also use
the common short-hand notation $\phi_0(z,n) \sim \phi_1(z,n)$ to abbreviate the asymptotic relation
$\phi_0(z,n) = \phi_1(z,n) (1+\oo(1))$ as $|z|\to\infty$
in some specified manner.

\begin{theorem}\label{thmbm}
Let $\tilde{n}\in\Z$, suppose $\theta_0(z,n)$, $\theta_1(z,n)$, $\phi_0(z,n)$, $\phi_1(z,n)$ are of growth order at most $s$ for some $s>0$ and $\phi_0(z,\tilde{n}) \sim \phi_1(z,\tilde{n})$ as $|z|\to\infty$ along some non-real rays dissecting the complex plane into sectors of opening angles less than $\nicefrac{\pi}{s}$.
Then the following properties are equivalent:
\begin{enumerate}
 \item We have $a_0(n) = a_1(n)$ for $n< \tilde{n}$, $b_0(n)=b_1(n)$ for $n\leq \tilde{n}$ and 
  \[
   \lim_{n\rightarrow-\infty} a_0(n)\Bigl(\phi_0(z,n)\phi_1(z,n+1) - \phi_0(z,n+1)\phi_1(z,n)\Bigr) = 0.
  \]
 %$\phi_0(z,n)=\phi_1(z,n)=0$ for $n< \tilde{n}$.
 \item For each $\delta>0$ there is an entire function $f(z)$ of growth order at most $s$ such that
  \[
  M_1(z)-M_0(z) = f(z) + \oo\left(\frac{1}{\phi_0(z,\tilde{n}+1)^{2}}\right),
  \]
 as $|z|\rightarrow\infty$ in the sector $|\im(z)|\geq \delta\,|\re(z)|$.
\end{enumerate}
\end{theorem}

\begin{proof}
If (i) holds, then by Remark~\ref{rem:uniqExp} and our assumptions 
\begin{equation}\label{eqnLBMsolTHETA}
 \phi_1(z,n)=\phi_0(z,n) \quad\text{and}\quad \theta_1(z,n) = \theta_0(z,n) -f(z)\phi_1(z,n), \quad z\in\C
\end{equation}
for all $n\leq \tilde{n}$, where $f(z)$ is some real entire function of growth order at most $s$. In particular, we obtain
\[
 \frac{\theta_1(z,\tilde{n}+1)}{a_0(\tilde{n})} = \frac{\theta_0(z,\tilde{n}+1) -f(z)\phi_0(z,\tilde{n}+1)}{a_1(\tilde{n})}, \quad
  \frac{\phi_1(z,\tilde{n}+1)}{a_0(\tilde{n})} =  \frac{\phi_0(z,\tilde{n}+1)}{a_1(\tilde{n})}
\]
and the asymptotics in Lemma~\ref{lemAsymM} show that
\begin{align*}
 M_1(z) - M_0(z) & = \frac{\theta_0(z,\tilde{n}+1)}{\phi_0(z,\tilde{n}+1)} - \frac{\theta_1(z,\tilde{n}+1)}{\phi_1(z,\tilde{n}+1)} + \OO\left(\frac{1}{z \phi_0(z,\tilde{n}+1)^{2}}\right) \\
                 & = f(z) + \OO\left(\frac{1}{z\phi_0(z,\tilde{n}+1)^{2}}\right),
\end{align*}
as $|z|\to\infty$ in any sector $|\im(z)| \ge \delta\, |\re(z)|$.

Now suppose property (ii) holds and for each fixed $n\leq \tilde{n}+1$ consider the entire function
\begin{equation}\label{eqnthmbmentfun}
 G_n(z) = \phi_1(z,n)  \theta_0(z,n) - \phi_0(z,n) \theta_1(z,n) - f(z)\phi_0(z,n) \phi_1(z,n), \quad z\in\C.
\end{equation}
Since away from the real axis this function may be written as
\begin{align*}
 G_n(z) & = \phi_1(z,n) \psi_0(z,n) - \phi_0(z,n) \psi_1(z,n) \\
      & \qquad\qquad + (M_1(z)-M_0(z) - f(z)) \phi_0(z,n) \phi_1(z,n), \quad z\in\C\backslash\R,
\end{align*}
it vanishes as $|z|\to\infty$ along our non-real rays. For the first two terms this
follows from \eqref{asympsi} together with our hypothesis that $\phi_0(\,\cdot\,,n)$ and $\phi_1(\,\cdot\,,n)$
have the same order of magnitude (in view of Lemma~\ref{lemPhiAs}). The last term tends to zero because of our assumption on the difference of the Weyl functions.
Moreover, by our hypothesis $G_n(z)$ is of growth order at most $s$ and thus we
can apply the Phragm\'en--Lindel\"of theorem (e.g., \cite[Sect.~6.1]{lev}) in the angles bounded by our rays.
This shows that $G_n(z)$ is bounded on all of $\C$.
By Liouville's theorem it must be constant and since it vanishes along a ray, it must be zero; that is,
\[
\phi_1(z,n) \theta_0(z,n) - \phi_0(z,n) \theta_1(z,n) = f(z)\phi_0(z,n)\phi_1(z,n), \quad n\le \tilde{n}+1,~z\in\C.
\]
Dividing both sides by $\phi_0(z,n)\phi_1(z,n)$ and taking differences shows
\[
\frac{\theta_0(z,n-1)}{\phi_0(z,n-1)} - \frac{\theta_0(z,n)}{\phi_0(z,n)} = \frac{\theta_1(z,n-1)}{\phi_1(z,n-1)} - \frac{\theta_1(z,n)}{\phi_1(z,n)}, \quad n\le \tilde{n}+1,~z\in\C\backslash\R,
\]
and, using $W_0(\theta_0,\phi_0)=W_1(\theta_1,\phi_1)=1$,
\[
a_0(n-1)\phi_0(z,n-1)\phi_0(z,n) = a_1(n-1)\phi_1(z,n-1)\phi_1(z,n), \quad n\le \tilde{n}+1,~z\in\C.
\]
Hence the Dirichlet and Neumann eigenvalues are equal (we know which is which by interlacing and the high energy asymptotics of $m_-(z,n)$) and the result
follows from \cite[Thm.~4.6]{ttr}.
\end{proof}

Observe that the implication (ii) $\Rightarrow$ (i) could also be proved under somewhat weaker conditions.
First of all the assumption on the growth of the entire functions $f(z)$ is only due to the use of the Phragm\'{e}n--Lindel\"{o}f principle.
Hence it would also suffice that for each $\eps>0$ we have
\[
\sup_{|z|=r_n} |f(z)| \leq B \E^{A r_n^{s+\eps}}, \quad n\in\N
\]
for some increasing sequence of positive reals $r_n\uparrow\infty$ and constants $A$, $B>0$.
Furthermore, for this implication to hold it would also suffice that $\phi_0(z,n)\asymp\phi_1(z,n)$ for one (and hence all) $n\in\Z$ as $|z|\rightarrow\infty$ along our non-real rays instead of $\phi_0(z,\tilde{n})\sim\phi_1(z,\tilde{n})$. 
 Here, the notation $\phi_0(z,n)\asymp\phi_1(z,n)$ is short hand for both, $\phi_0(z,n) \phi_1(z,n)^{-1} = \OO(1)$ and $\phi_1(z,n) \phi_0(z,n)^{-1} = \OO(1)$ to hold.

While at first sight it might look like the condition on the asymptotics of the solutions $\phi_j(z,n)$ requires knowledge
about them, this is not the case, since the high energy asymptotics will only involve some qualitative information
on behavior of the coefficients as $n\to-\infty$.
Next, the appearance of the additional freedom of the function $f(z)$ just reflects the fact that we only ensure the same normalization
for the solutions $\phi_0(z,n)$ and $\phi_1(z,n)$ but not for $\theta_0(z,n)$ and $\theta_1(z,n)$ (cf.\ Remark~\ref{rem:uniqExp}).

\begin{corollary}\label{corbm}
Suppose $\theta_0(z,n)$, $\theta_1(z,n)$, $\phi_0(z,n)$, $\phi_1(z,n)$ are of growth order at most $s$ for some $s>0$ and  $\phi_0(z,n) \asymp \phi_1(z,n)$ for one $n\in\Z$ as $|z|\to\infty$ along some non-real rays dissecting the complex plane into sectors of opening angles less than $\nicefrac{\pi}{s}$.
If 
\begin{equation}\label{eqncorbm}
 M_1(z) - M_0(z) = f(z), \quad z\in\C\backslash\R,
\end{equation}
for some entire function $f(z)$ of growth order at most $s$, then $H_0=H_1$.
\end{corollary}

\begin{proof}
By (the remark after the proof of) Theorem~\ref{thmbm} it remains to show that $H_0$ and $H_1$ have the same boundary condition near $+\infty$. If there are one at all,
then $H_0$ and $H_1$ have a common eigenvalue $\lam$ and the claim follows since $\phi_0(\lam,\cdot\,)$ and $\phi_1(\lam,\cdot\,)$ are linearly dependent and hence
satisfy the boundary conditions of both $H_0$ and $H_1$. 
%By the proof of Theorem~\ref{thmbm} we have $\phi_1(z,n)=\phi_0(z,n)$ as well as $\theta_1(z,n)=\theta_0(z,n)-f(z)\phi_0(z,n)$ implying $\psi_1(z,n)= \psi_0(z,n)$ from which the claim follows.
\end{proof}

\section{Uniqueness results for operators with discrete spectra}
\label{sec:urds}

In this section we want to investigate when the spectral measure determines the coefficients $a$ and $b$ for operators with purely discrete spectrum.
To this end note that the uniqueness results for the singular Weyl function from the previous sections do not immediately imply
such results. Indeed, if $\rho_0=\rho_1$ then by Theorem~\ref{IntR} the difference of the corresponding singular Weyl functions is an entire function.
However, for the application of Corollary~\ref{corbm} we would need some bound on the growth order of this function. Fortunately,
in the case of purely discrete spectrum with finite convergence exponent, a refinement of the arguments in the proof
of Theorem~\ref{thmbm} shows that this growth condition can be avoided. This can be shown literally as in \cite[Cor.~5.1]{et}.

\begin{corollary}\label{corbmdis}
Suppose $\phi_0(z,n)$, $\phi_1(z,n)$ are of growth order at most $s$ for some $s>0$ and  $\phi_0(z,n) \asymp \phi_1(z,n)$ for one $n\in\Z$ as $|z|\to\infty$
along some non-real rays dissecting the complex plane into sectors of opening angles less than $\nicefrac{\pi}{s}$.
Furthermore, assume that $H_0$ and $H_1$ have purely discrete spectrum with convergence exponent at most $s$.
If 
\begin{equation}
 M_1(z) - M_0(z) = f(z), \quad z\in\C\backslash\R,
\end{equation}
for some entire function $f(z)$, then $H_0=H_1$.
\end{corollary}

Now the lack of a growth restriction in Corollary~\ref{corbmdis} implies a corresponding uniqueness result for the spectral measure. Again this follows literally as in \cite[Thm.~5.2]{et}. 

\begin{theorem}\label{thmSpectFuncDisc}
Suppose that $\phi_0(z,n)$, $\phi_1(z,n)$ are of growth order at most $s$ for some $s>0$ and $\phi_0(z,n)\asymp\phi_1(z,n)$ for one $n\in\Z$ as $|z|\rightarrow\infty$ along
some non-real rays dissecting the complex plane into sectors of opening angles less than $\nicefrac{\pi}{s}$.  
Furthermore, assume that $H_0$ and $H_1$ have purely discrete spectrum with convergence exponent at most $s$. If the corresponding spectral
measures $\rho_0$ and $\rho_1$ are equal, then we have $H_1=H_0$.
\end{theorem}

It is also worth while noting that in the case of discrete spectra, the spectral measure is uniquely determined by the eigenvalues $\lam\in\sigma(H)$ together with the corresponding norming constants
\begin{equation}
 \gamma_\lam^2 = \sum_{m\in\Z} \phi(\lam,m)^2, \quad \lam\in\sigma(H)
\end{equation}
since in this case we have
\begin{equation}
 \rho = \sum_{\lam\in\sigma(H)} \gamma_\lam^{-2} \delta_{\lam},
\end{equation}
where $\delta_{\lam}$ is the Dirac measure at the point $\lam$.

As another application we are also able to proof a generalization of the famous Hochstadt--Liebermann-type uniqueness result.
To this end let us consider a Jacobi operator $H$ whose spectrum is purely discrete and has convergence exponent (at most) $s$.
Since the Jacobi operator with an additional Dirichlet boundary condition at zero is a rank one perturbation of $H$ we conclude
that the  convergence exponents of the spectra of $H_{-}$ and $H_{+}$ are at most $s$ as well. Hence by Theorem~\ref{thm:IOphiev}
there exist real entire solutions $\phi(z,n)$ and $\chi(z,n)$ of growth order at most $s$ which are in the domain of $H$ near $-\infty$ and $+\infty$, respectively.

\begin{theorem}\label{thmHL}
Suppose $H_0$ is a Jacobi operator with purely discrete spectrum of finite convergence exponent $s> 0$. Let $\phi_0(z,n)$ and $\chi_0(z,n)$ be
entire solutions of growth order at most $s$ which lie in the domain of $H_0$ near $-\infty$ and $+\infty$, respectively, and suppose there is an $\tilde{n}\in\Z$ such that
\begin{equation}\label{eqeahl}
\frac{\chi_0(z,\tilde{n})}{\phi_0(z,\tilde{n})}=\oo(1)
\end{equation}
as $|z|\rightarrow\infty$ along some non-real rays dissecting the complex plane into sectors of opening angles less than $\nicefrac{\pi}{s}$.

Then every other isospectral Jacobi operator $H_1$ with $a_1(n)=a_0(n)$ for $n<\tilde{n}-1$, $b_1(n) = b_0(n)$ for $n < \tilde{n}$ and
which is associated with the same boundary condition at $-\infty$ (if any) is equal to $H_0$.
\end{theorem}

\begin{proof}
Start with some solutions $\phi_j(z,n)$, $\chi_j(z,n)$ of growth order at most $s$ and note that we can choose $\phi_1(z,n) = \phi_0(z,n)$ for $n< \tilde{n}$ since $H_1$ and $H_0$
are associated with the same boundary condition at $-\infty$ (if any). Moreover, by Lemma~\ref{IOphias} we have $\phi_0(z,n) \asymp \phi_1(z,n)$ as $|z|\to\infty$
along every non-real ray, even for $n\geq \tilde{n}$. 
 Next note that the zeros of the Wronskians $W_j(\phi_j,\chi_j)$ are
precisely the eigenvalues of $H_j$ and thus, by assumption, are equal. Hence by the Hadamard factorization theorem $W_1(\phi_1,\chi_1)= \E^{g} W_0(\phi_0,\chi_0)$
for some polynomial $g$ of degree at most $s$. Since we can absorb this factor in $\chi_1(z,n)$, we can assume $g=0$ without loss of generality. Hence we have
\begin{align*}
 1= & \frac{W_0(\phi_0,\chi_0)}{W_1(\phi_1,\chi_1)} = \frac{a_0(\tilde{n}-1)(\phi_0(z,\tilde{n}-1)\chi_0(z,\tilde{n}) - \phi_0(z,\tilde{n})\chi_0(z,\tilde{n}-1))}{a_1(\tilde{n}-1)(\phi_1(z,\tilde{n}-1)\chi_1(z,\tilde{n}) - \phi_1(z,\tilde{n})\chi_1(z,\tilde{n}-1))} \\
      =& \frac{\phi_0(z,\tilde{n}-1)}{\phi_1(z,\tilde{n}-1)} \frac{\chi_0(z,\tilde{n}-1)}{\chi_1(z,\tilde{n}-1)} \left(\frac{a_0(\tilde{n}-1)\chi_0(z,\tilde{n})}{\chi_0(z,\tilde{n}-1)} - \frac{a_0(\tilde{n}-1)\phi_0(z,\tilde{n})}{\phi_0(z,\tilde{n}-1)}\right) \times\\
     &\times  \left(\frac{a_1(\tilde{n}-1)\chi_1(z,\tilde{n})}{\chi_1(z,\tilde{n}-1)} - \frac{a_1(\tilde{n}-1)\phi_1(z,\tilde{n})}{\phi_1(z,\tilde{n}-1)}\right)^{-1}
\end{align*}
and by virtue of the well-known asymptotics (see~\cite[Lemma~6.6]{tjac}) for $j=0,1$ 
\[
 \frac{a_j(\tilde{n}-1)\chi_j(z,\tilde{n})}{\chi_j(z,\tilde{n}-1)} = \OO(z^{-1}) \quad\text{and}\quad  \frac{a_j(\tilde{n}-1)\phi_j(z,\tilde{n})}{\phi_j(z,\tilde{n}-1)} = z - b(\tilde{n}-1) + \OO(z^{-1})
\]
as $|z|\rightarrow\infty$ along any non-real ray we conclude 
\[
\chi_1(z,\tilde{n}-1) = \chi_0(z,\tilde{n}-1) (1+\OO(z^{-2}))
\]
 as $|z|\rightarrow\infty$ along non-real rays. Equality of the Wronskians also implies
\[
\chi_1(z,n) = \chi_0(z,n) + F(z) \phi_0(z,n), \quad n < \tilde{n}
\]
for some entire function $F(z)$ of growth order at most $s$. Moreover, our assumption \eqref{eqeahl} implies that
\[
 F(z) = \frac{\chi_1(z,\tilde{n}-1)-\chi_0(z,\tilde{n}-1)}{\phi_0(z,\tilde{n}-1)} = \frac{\chi_0(z,\tilde{n}-1)}{\phi_0(z,\tilde{n}-1)}\left(\frac{\chi_1(z,\tilde{n}-1)}{\chi_0(z,\tilde{n}-1)} -1\right)
\]
vanishes along our rays and thus it must be identically zero by the Phragm\'en--Lindel\"of theorem. Finally, choosing
$\theta_j(z,n)$ such that $\theta_1(z,n) = \theta_0(z,n)$ for $n< \tilde{n}$, this implies that the associated singular Weyl functions
are equal and the claim follows from Corollary~\ref{corbmdis}.
\end{proof}

Note that by \eqref{IOphias} the growth of $\phi_0(\,\cdot\,,\tilde{n})$ will increase as $\tilde{n}$ increases while (by reflection) the growth of $\chi_0(\,\cdot\,,\tilde{n})$
will decrease. In particular, if \eqref{eqeahl} holds for $\tilde{n}$ it will hold for any $n_1>\tilde{n}$ as well. Moreover, one may change knowledge of the coefficient $b_1(\tilde{n})=b_0(\tilde{n})$ for the stronger asymptotics $\oo(z^{-1})$ of the quotient in \eqref{eqeahl}.

\section{Jacobi operators and de Branges spaces}

For each $n\in\Z$ we consider the entire function
\begin{equation}
E(z,n)= \phi(z,n) + \I a(n)\phi(z,n+1), \quad z\in\C
\end{equation}
which satisfies
\[
 \frac{E(z,n) E^\#(\zeta^*,n) - E(\zeta^*,n) E^\#(z,n)}{2\I (\zeta^* -z)} = \sum_{m=-\infty}^n \phi(\zeta,m)^* \phi(z,m), \quad \zeta,\,z\in\C^+,
\]
where $F^\#(z) = F(z^\ast)^\ast$, $z\in\C$. 
 In particular, taking $\zeta=z$ this shows that $E(\,\cdot\,,n)$ is a de Branges function.
Moreover, note that $E(\,\cdot\,,n)$ does not have any real zero $\lam$, since otherwise both, $\phi(\lam,n)$ and $\phi(\lam,n+1)$ would vanish.
With $B(n)$ we denote the de Branges space associated with the de Branges function $E(\,\cdot\,,n)$ endowed with the inner product
\[
 \dbspr{F}{G}_{B(n)} = \frac{1}{\pi} \int_\R \frac{F(\lam)^* G(\lam)}{|E(\lam,n)|^2} d\lam = \frac{1}{\pi}\int_\R \frac{F(\lam)^*G(\lam)}{\phi(\lam,n)^2 + a(n)^2\phi(\lam,n+1)^2}d\lam
\]
for $F$, $G\in B(n)$.
The reproducing kernel $K(\,\cdot\,,\cdot\,,n)$ of this space is given by
\begin{equation}\label{eqndBschrRepKer}
 K(\zeta,z,n) = \sum_{m=-\infty}^n \phi(\zeta,m)^* \phi(z,m), \quad \zeta,\,z\in\C.
\end{equation}

\begin{theorem}\label{thmdBschrBT}
 For every $n\in\Z$ the transformation $f\mapsto\hat{f}$ is unitary from $\ell^2(-\infty,n]$ onto $B(n)$. In particular,
 \begin{equation}\label{eqnBn}
  B(n) = \big\lbrace  \hat{f} \,\big|\, f\in \ell^2(-\infty,n] \big\rbrace = \ol{\lspan\{\phi(\,\cdot\,,m) \,|\, m\le n\}}.
 \end{equation}
 Here we identify $\ell^2(-\infty,n]$ with the subspace of sequences in $\ell^2(\Z)$ which vanish on $m>n$.
\end{theorem}
 
\begin{proof}
 For each $\lam\in\R$ consider the function
 \[
  f_\lam(m) = \begin{cases} \phi(\lam,m), & m \le n, \\ 0,  & m>n. \end{cases}
 \]
 The transforms of these functions are given by
 \[
  \hat{f}_\lam(z) = \sum_{m=-\infty}^n \phi(\lam,m) \phi(z,m) = K(\lam,z,n), \quad z\in\C.
 \]
 In particular, this shows that the transforms of the functions $f_\lam$, $\lam\in\R$, lie in $B(n)$. Moreover, we have for all $\lam_1$, $\lam_2\in\R$
 \begin{align*}
  \spr{f_{\lam_1}}{f_{\lam_2}} & = \sum_{m=-\infty}^n \phi(\lam_1,m) \phi(\lam_2,m) = K(\lam_1,\lam_2,n) \\ 
  &  = \dbspr{K(\lam_1,\cdot\,,n)}{K(\lam_2,\cdot\,,n)}_{B(n)}.
 \end{align*}
 Hence our transform is an isometry on the linear span $D$ of all functions $f_\lam$, $\lam\in\R$.
  But this span is dense in $\ell^2(-\infty,n] $ since it contains the eigenfunctions of the operator $H$ restricted to $m\leq n$ with a Dirichlet boundary condition at $n+1$.
 Moreover, the linear span of all transforms $K(\lam,\cdot\,,n)$, $\lam\in\R$, is dense in $B(n)$. Indeed, each $F\in B(n)$ such that
 \[
  0 = \dbspr{K(\lam,\cdot\,,n)}{F}_{B(n)} = F(\lam), \quad \lam\in\R
 \]
 vanishes identically. Thus our transformation restricted to $D$ uniquely extends to a unitary map $V$ from $\ell^2(-\infty,n]$ onto $B(n)$.
 Finally, since $f\mapsto \hat{f}(z)$ as well as $f\mapsto V f(z)$ are continuous on $\ell^2(-\infty,n]$ for each $z\in\C$, we infer that this extension coincides with our transformation. 
\end{proof}

As an immediate consequence of Theorem~\ref{thmdBschrBT} and the fact that our transformation from~\eqref{eqhatf} extends 
 to a unitary map from $\ell^2(\Z)$ onto $L^2(\R,d\rho)$, we get the following corollary.

\begin{corollary}\label{cordBschrembL2}
For each $n\in\Z$ the de Branges space $B(n)$ is isometrically embedded in $L^2(\R,d\rho)$, that is
\begin{equation}
 \int_\R |F(\lam)|^2 d\rho(\lam) = \|F\|^2_{B(n)}, \quad F\in B(n).
\end{equation}
Moreover, the union of the spaces $B(n)$, $n\in\Z$, is dense in $L^2(\R,d\rho)$, i.e.,
\begin{equation}\label{eqndBschrdense}
 \overline{\bigcup_{n\in\Z} B(n)} = L^2(\R,d\rho).
\end{equation}
\end{corollary}

Clearly the de Branges spaces $B(n)$, $n\in\Z$, are totally ordered and strictly increasing in the sense that
$B(n) \subsetneq B(\tilde{n})$ for $n<\tilde{n}$.
More precisely, Theorem~\ref{thmdBschrBT} shows that $B(n)$ actually has codimension one in $B(n+1)$. 

For the proof of our main result we will need the ordering theorem due to de Branges. 
In order to state it let $E_0$, $E_1$ be two de Branges functions and $B_0$, $B_1$ be the corresponding de Branges spaces.

\begin{theorem}[\cite{dBbook}, Theorem~35]\label{thmdBOrdering}
Suppose $B_0$, $B_1$ are isometrically embedded in $L^2(\R,d\rho)$, for some Borel measure $\rho$ on $\R$. 
If $E_0/E_1$ is of bounded type in the upper com\-plex half-plane and has no real zeros or singularities, then $B_0$ contains $B_1$ or $B_1$ contains $B_0$.
\end{theorem}

Moreover, one has the following simple converse statement.

\begin{lemma}[\cite{je}, Lemma 2.2]\label{lemdBordcon}
If $B_0$ contains $B_1$ or $B_1$ contains $B_0$, then $E_0/E_1$ is of bounded type in the upper complex half-plane.
\end{lemma}

Now let $H_0$, $H_1$ be two Jacobi operators with separated boundary conditions. Suppose there are corresponding
nontrivial real entire solutions $\phi_0(z,n)$, $\phi_1(z,n)$ which are square integrable near the left endpoint and satisfy the
boundary condition there, if any. In obvious notation we denote all quantities corresponding to $H_0$ and $H_1$ with an additional subscript. 
We say $H_0$ and $H_1$ are equal up to a shift if there is some $k\in\Z$ such that
\[
H_0=S^{-k}H_1\,S^k,
\]
where $S: \ell^2(\Z)\to\ell^2(\Z)$, $f(m)\mapsto f(m+1)$, is the usual shift operator.

\begin{theorem}\label{thmdBuniqS}
Suppose there is a real entire function $\E^g$ without zeros such that
\begin{equation}\label{eqnquotE1E2}
  \E^{g(z)} \frac{E_0(z,n_0)}{E_1(z,n_1)}, \quad z\in\C^+
\end{equation}
is of bounded type for some $n_0$, $n_1\in\Z$. If $d\rho_0(\lam)=\E^{-2g(\lam)} d\rho_1(\lam)$, then $H_0$ and $H_1$ are equal up to a shift.
\end{theorem}

\begin{proof}
First of all note that without loss of generality we may assume that $g$ vanishes identically, since otherwise we replace $\phi_0(z,n)$ with $\E^{g(z)} \phi_0(z,n)$.
Moreover, because of Lemma~\ref{lemdBordcon} the function in~\eqref{eqnquotE1E2} is of bounded type for all points $n_0$, $n_1\in\Z$.
 Hence Corollary~\ref{cordBschrembL2} and Theorem~\ref{thmdBOrdering} (note that~\eqref{eqnquotE1E2} has no real zeros or singularities because $E_0(\,\cdot\,,n_0)$ and $E_1(\,\cdot\,,n_1)$ do not have real zeros) 
  imply that either $B_0(n_0)$ is contained in $B_1(n_1)$ or $B_1(n_1)$ is contained in $B_0(n_0)$. 
 Without loss of generality we may assume
$B_0(n_0) \subseteq B_1(n_1)$. Moreover, by \eqref{eqnBn} and \eqref{eqndBschrdense} we can further increase $n_0$ such that
$B_0(n_0) \subseteq B_1(n_1) \subsetneq B_0(n_0+1)$. Then clearly $B_0(n_0) = B_1(n_1)$ and we set $k=n_1-n_0$. Moreover,
since $B_j(n)$ has codimension one in $B_j(n+1)$, we obtain by induction $B_0(n)=B_1(n+k)$ for all $n\in\Z$.
Now \eqref{eqndBschrRepKer} shows $\phi_0(z,n)= \sigma_n \phi_1(z,n+k)$ for some sequence $\sig_n\in\{\pm 1\}$, which is actually independent of $n$ in view of Lemma~\ref{lemPhiAs}.
 Similarly as in the proof of Theorem~\ref{thmbm}, this ensures that the coefficient sequences and the boundary conditions at $-\infty$ are the same up to a shift by $k$. 
 Finally, the fact that the boundary conditions at $\infty$ are the same follows as in the proof of Corollary~\ref{corbm}.
\end{proof}

 Note that the quotient~\eqref{eqnquotE1E2} in Theorem~\ref{thmdBuniqS} is of bounded type if and only if 
 \[
  \E^{g(z)} \frac{\phi_0(z,n_0)}{\phi_1(z,n_1)}, \quad z\in\C^+
 \]
 % since E_0/E_1 = E_0/\phi_0 * \phi_0/\phi_1 * \phi_1/E_1
 is of bounded type. 
 We conclude this section by showing that this condition holds if the solutions $\phi_0(z,n)$, $\phi_1(z,n)$ satisfy some growth condition. 
Therefore recall that an entire function $F$ belongs to the Cartwright class $\mathcal{C}$ if it is of finite exponential type and the logarithmic integral
\[
 \int_\R \frac{\ln^+|F(\lam)|}{1+\lam^2}d\lam < \infty, 
\]
 where $\ln^+$ is the positive part of the natural logarithm. In particular, note that the class $\mathcal{C}$ contains all entire functions of growth order less than one.
Now a theorem of Krein~\cite[Theorem~6.17]{rosrov}, \cite[Section~16.1]{lev} states that the class $\mathcal{C}$ consists of all entire functions which are of bounded type in the upper and in the lower complex half-plane.
Since the quotient of two functions of bounded type is of bounded type itself, this immediately yields the following uniqueness result.

\begin{corollary}\label{cordBuniqScor}
 Suppose that $\phi_0(\,\cdot\,,n_0)$ and $\phi_1(\,\cdot\,,n_1)$ belong to the Cart\-wright class $\mathcal{C}$ for some $n_0$, $n_1\in\Z$.
 If $\rho_0=\rho_1$, then $H_0$ and $H_1$ are equal up to a shift.
\end{corollary}

\bigskip
\noindent
{\bf Acknowledgments.}
We thank Fritz Gesztesy for hints with respect to the literature and the anonymous referee for suggestions improving the presentation
of the material. G.T.\ gratefully acknowledges the stimulating
atmosphere at the Isaac Newton Institute for Mathematical Sciences in Cambridge
during October 2011 where parts of this paper were written as part of  the international
research program on Inverse Problems.

\end{document}